\documentclass[reqno,centertags, 12pt, draft]{amsart}

\usepackage{amsmath,amsthm,amscd,amssymb}

\usepackage{latexsym}

\usepackage{mathrsfs}

\newcommand{\bbN}{{\mathbb{N}}}

\newcommand{\bbR}{{\mathbb{R}}}

\newcommand{\bbC}{{\mathbb{C}}}

\newcommand{\no}{\nonumber}

\newcommand{\supp}{\text{\rm{supp}}}

\newcommand{\beq}{\begin{equation}}

\newcommand{\eeq}{\end{equation}}

\newcommand{\ba}{\begin{align}}

\newcommand{\ea}{\end{align}}

\allowdisplaybreaks

\numberwithin{equation}{section}

\newtheorem{theorem}{Theorem}[section]

\newtheorem{proposition}[theorem]{Proposition}

\newtheorem{lemma}[theorem]{Lemma}

\newtheorem{corollary}[theorem]{Corollary}

\theoremstyle{definition}

\newtheorem{definition}[theorem]{Definition}

\theoremstyle{remark}

\newtheorem{remark}{Remark}[section]

\date{}
\begin{document}
\title[Level repulsion and s.c. spectrum]{Level repulsion for Schr\"odinger operators with singular continuous
spectrum}

\author[J.~Breuer and D.~Weissman]{Jonathan~Breuer and Daniel~Weissman}

\thanks{Institute of Mathematics, The Hebrew University of Jerusalem, Jerusalem, 91904, Israel.
E-mail: jbreuer@math.huji.ac.il; daniel.weissman@mail.huji.ac.il. Supported in part by The Israel Science
Foundation (Grant No. 1105/10)}

\maketitle
\sloppy
\begin{abstract}
We describe a family of half-line continuum Schr\"odinger operators
with purely singular continuous spectrum on $(0,\infty)$, exhibiting asymptotic strong
level repulsion (known as clock behavior). This follows from the convergence of the renormalized continuum Christoffel-Darboux
kernel to the sine kernel.
\end{abstract}

\section{Introduction}

The problem of understanding the asymptotics of finite-volume level spacings for Schr\"odinger operators has been receiving a considerable amount of attention in recent years (\cite{aw, als, brGro, brSine, bfs, findley, andrei, gk, kotani-nakano, kritchevski, LaSic, LeLu, Lub1, Lub2, maltsev, mastroianni, minami, molchanov, nakanoBeta, simonClock, sel, totikUniversality, valko-virag-noise} is a small subset of relevant references). While results in the multidimensional setting were obtained predominantly for random operators, in the one-dimensional case both deterministic and random operators were studied.

Particularly interesting in this context is the problem of understanding how the asymptotics of level spacings are connected to continuity properties of the spectral measures. Known results indicate a certain rough correspondence between asymptotic repulsion and continuity of the spectral measures. Perhaps the most studied case is that of the localized regime in the Anderson model, where it has been shown under various conditions that the rescaled finite-volume eigenvalue process converges to a Poisson process on the line (see, e.g., \cite{gk, minami, molchanov}). Results for one-dimensional operators with random decaying potentials \cite{bfs, kotani-nakano, kritchevski, nakanoBeta} seem to indicate a pattern whereby greater continuity of the spectral measure corresponds to greater repulsion. Deterministic results establishing asymptotic regular spacing for absolutely continuous measures \cite{als, findley, LaSic, LeLu, Lub1, Lub2, maltsev, mastroianni, simonClock, sel, totikUniversality}  work in the same vein.

Our aim in the present paper is to show that the situation is more subtle than what may be thought in light of the discussion above. We shall present a family of half-line Schr\"odinger operators with purely singular continuous spectrum on the positive half-line, whose finite-volume eigenvalues display clock asymptotic behavior (see Definition \ref{def:clock} below), which is a very strong form of repulsion. This shows that strong repulsion should not be associated only with absolutely continuous spectrum.

For discrete Schr\"odinger operators, an analogous result was obtained by one of us \cite{brSine} in the context of studying conditions on measures guaranteeing universal behavior of the associated Christoffel-Darboux (CD) kernel. In particular, in \cite{brSine} clock spacing is a consequence of the convergence of the CD kernel to the sine kernel. We shall obtain our result here by exploiting the analogy between the discrete and continuous case.

To fix terminology and notation, a Schr\"odinger operator acting on $\mathbb{R}^{+}$ is an operator of the form
\beq \no
H=\Delta+V
\eeq
where $\Delta=-\frac{\textrm{d}^{2}}{\textrm{d} x^{2}}$ denotes the Laplacian,
and the operator defined by multiplication by $V\left(x\right)$ is the potential (we shall soon specify conditions on $V$).
We will assume Neumann boundary conditions throughout, i.e.\
\beq \no
u\left(0\right)=1\,\,\,\mbox{and}\,\,\, u'\left(0\right)=\frac{\textrm{d}}{\textrm{d} x}u\left(x\right)|_{x=0}=0.
\eeq

By the spectral theorem there exists a unique measure, $\mu$, for which $H$ is unitarily equivalent
to the operator of multiplication by $x$ on $L^{2}\left(\mathbb{R},\textrm{d}\mu\left(x\right)\right)$.
We call $\mu$ the spectral measure associated with the operator
$H$.

Given $H$ as above, we can restrict it to intervals $\left[0,L\right]$ and consider the operator
\beq \no
H^{L}=\Delta+V|_{\left[0,L\right]}
\eeq
acting on $L^{2}\left(0,L\right)$, this time with Neumann boundary
conditions at $L$ as well as at $0$, i.e., with the \textit{additional}
condition
\beq \no
u'\left(L\right)=0.
\eeq

The spectrum of $H^{L}$ is a discrete set of eigenvalues $\left\{ \xi_{j}^{L}\right\} $. We are interested in the connection between properties of $\mu$ and asymptotic properties (as $L \rightarrow \infty$) of the spacings between the $\xi_j^L$'s. The asymptotic density of these eigenvalues is measured by the \emph{density of states measure}, $\nu$, defined as the weak limit (if it exists) of  the normalized eigenvalue counting measure. Namely, letting
\begin{equation} \label{eq:Density_of_states_definition}
\nu_{L}=\frac{1}{L}\sum_{j}\delta_{\xi_{j}^{L}},
\end{equation}
$\nu=\mbox{w\ensuremath{-\mbox{lim}}}_{L\rightarrow\infty}\nu_{L}$ (if it exists). Assuming $\nu$ exists and is absolutely continuous w.r.t.\ Lebesgue measure, we write $\textrm{d} \nu(\xi)=\rho(\xi) \textrm{d}\xi$. We call $\rho$ the \emph{density of states}. For the free Schr\"odinger operator ($H_0=\Delta$)
\begin{equation} \label{eq:Density_Of_States}
\rho\left(\xi\right)=\frac{1}{2\pi}\xi^{-\frac{1}{2}}
\end{equation}
(see Example 8.1 in Section 2.8 of \cite{BerezShub}).
\begin{definition} \label{def:clock}
Let $I\subseteq\mathbb{R}^{+}$ be a closed interval, and let $\xi^{*}\in I$.
For each $L>0$, reenumerate the eigenvalues $\{\xi_j^L\}$ around $\xi^{*}$ as follows:
\beq \no
...<\xi_{-1}^{L}\left(\xi^{*}\right)<\xi^{*}\leq\xi_{0}^{L}\left(\xi^{*}\right)<\xi_{1}^{L}\left(\xi^{*}\right)<...
\eeq
Following \cite{als}, we say there is \emph{strong clock behavior} at $\xi^{*}$, if the density of states exists
and for each fixed $n$,
\beq \label{eq:strong-clock}
\underset{L\rightarrow\infty}{\lim}\, L\left(\xi_{n+1}^{L}\left(\xi^{*}\right)-\xi_{n}^{L}\left(\xi^{*}\right)\right)\rho\left(\xi^{*}\right)=1
\eeq
We say there is \emph{uniform clock behavior} on the interval  $I\subseteq\mathbb{R}^{+}$,
if the above limit is uniform on $I$ for each fixed $n$.
\end{definition}

\begin{remark}
Originally, the concept was defined for operators arising in the theory of orthogonal
polynomials on the unit circle \cite{simonClock}. On the unit circle, the equally-spaced
eigenvalues appear as marks on a clock -- hence the name.
\end{remark}

\begin{definition}
If $u\left(\xi,x\right)$ is the unique solution of the equation
\beq \label{eq:ef_equation}
-\frac{\textrm{d}^{2}}{\textrm{d} x^{2}}u\left(\xi,x\right)+V\left(x\right)u\left(\xi,x\right)=\xi\cdot u\left(\xi,x\right)
\eeq
with Neumann boundary conditions, the \emph{continuous Christoffel-Darboux
(CD) kernel} at $L$ is
\begin{equation} \label{eq:CD-Kernel-Definition}
S_{L}\left(\xi,\zeta\right)=\int_{0}^{L}u\left(\xi,r\right)u\left(\zeta,r\right)dr
\end{equation}
\end{definition}

This object, introduced in \cite{maltsev}, is the continuous analog of the classical CD
kernel. The CD kernel arises naturally in the study of orthogonal polynomials and
has a wide range of applications (see \cite{simonCD}
for a survey). In particular, the phenomenon of \emph{universality} in random matrix theory is intimately connected with the fact that in many cases the rescaled CD kernel converges to the sine kernel (\cite{deift}).

The significance of this, in our case, lies in the fact that
if $u'\left(\xi,L\right)=0$, then $u'\left(\zeta,L\right)=0$ iff $S_{L}\left(\xi,\zeta\right)=0$. Since the zeros of $u'(\cdot,L)$ are the eigenvalues of $H^L$, this means there is a connection between the asymptotic properties of $S_{L}\left(\xi+\frac{a}{L},\xi+\frac{b}{L}\right)$,
and the small-scale behavior of the $\xi_j^L$'s
around $\xi$, as $L\longrightarrow\infty$.

We shall prove uniform clock behavior for the operators under consideration by showing that the associated CD kernel satisfies
\begin{equation} \label{eq:SineKernelAsymptotics}
\frac{S_{L}\left(\xi+\frac{a}{L},\xi+\frac{b}{L}\right)}{S_{L}\left(\xi,\xi\right)}\underset{L\rightarrow\infty}{\longrightarrow}\frac{\sin\left(\pi\cdot\rho\left(\xi\right)\left(b-a\right)\right)}{\pi\cdot\rho\left(\xi\right)\left(b-a\right)}
\end{equation}
uniformly for $\xi$ in compact subsets of $\mathbb{R}^+$ and $a,b$ in compact subsets of the strip $|\textrm{Im}z| \leq 1$.
The fact that this, together with existence of the density of states (which needs to be established separately), implies uniform clock behavior in the discrete case is known as the Freud-Levin Theorem \cite{freud, LeLu, simonCD}. The proof given in \cite[Section 23]{simonCD} translates directly to the continuum case and so

\begin{proposition} \label{thm:USKA-implies-clock}
For an operator $H=\Delta+V$ on $L^{2}\left(\mathbb{R}^{+}\right)$,
assume that \eqref{eq:SineKernelAsymptotics} holds uniformly for $\xi$ in compact subsets of $\mathbb{R}^+$ and $a,b$ in compact subsets of the strip $|\textrm{Im}z| \leq 1$. Assume moreover that the density of states for $H$ exists and is absolutely continuous with respect to Lebesgue measure. Then uniform clock behavior follows on any compact interval $I\subseteq\left(0,\infty\right)$.
\end{proposition}

In the case of discrete CD kernels \eqref{eq:SineKernelAsymptotics} has been proven under a wide range of conditions.
However, except for \cite{brSine} (and this work), it was
always assumed that the spectral measure was absolutely continuous in a neighborhood
of the point under consideration. That \eqref{eq:SineKernelAsymptotics} implies clock behavior in
the discrete case was discovered by Freud \cite{freud},
and rediscovered by Levin and Lubinsky \cite{LeLu}.
Additional results were obtained by Lubinsky in \cite{Lub1, Lub2},
and in greater generality, by Simon, \cite{sel}, Findley \cite{findley}, and Totik \cite{totikUniversality}. Avila, Last and Simon,
in \cite{als}, proved \eqref{eq:SineKernelAsymptotics} and clock behavior for ergodic
Jacobi matrices with a.c.\ spectrum.

To the best of our knowledge, Maltsev's paper \cite{maltsev} is the only
prior work dealing with \eqref{eq:SineKernelAsymptotics} in the continuous case, where it was used to prove clock asymptotics for eigenvalues of perturbed periodic Schr\"odinger operators.

We are now ready to take a closer look at the potentials
we consider in this work. Given a sequence of real numbers, $\left\{ \lambda_{n}\right\} _{n=1}^{\infty}$
such that $\lambda_{n}\longrightarrow0$, a sequence of positive numbers
$\left\{ N_{n}\right\} _{n=1}^{\infty}$ such that $\frac{N_{n}}{N_{n+1}}\longrightarrow0$,
and a bounded, non-negative, compactly supported function $W\left(x\right)$
(which we think of as a recurring perturbation of varying amplitude,
where the amplitudes are determined by the $\lambda_{n}$), we define
a so-called Pearson potential (see \cite{kls}),
\begin{equation} \label{eq:Pearson_Potential}
V\left(x\right)=\sum_{n=1}^{\infty}\lambda_{n}W\left(x-N_{n}\right)
\end{equation}
Thus, if we fix boundary conditions as above, there is an associated
self-adjoint Schr\"odinger operator,
\begin{equation}
H=\Delta+V.\label{eq:Schroedinger_Operator}
\end{equation}

Sparse potentials were first introduced by Pearson in \cite{pearson}
in 1978, in the construction of the first explicit examples of Schr\"odinger
operators exhibiting purely singular continuous spectrum. The following theorem of Kiselev, Last and Simon
extends the original Pearson result:
\begin{theorem}[Theorem 1.6 of \cite{kls}] \label{thm:Kiselev_Last_Simon_Thm}
Let $V$ be a Pearson potential.
If $\sum_{n=1}^{\infty}\lambda_{n}^{2}<\infty$ $($resp.\ $\sum_{n=1}^{\infty}\lambda_{n}^{2}=\infty$$)$,
the spectrum of the operator $\Delta+V$ is purely absolutely continuous
on $\mathbb{R}^{+}$ $($resp.\ purely singular continuous$)$, for
any boundary condition at $0$.
\end{theorem}

We are now ready to state our main result.
\begin{theorem} \noindent \label{thm:Main_Result}
Let $W\left(x\right)$ be a smooth,
non-negative function which is supported on $\left[0,1\right]$, and
let $\left\{ \lambda_{n}\right\} _{n=1}^{\infty}$ be a sequence such
that $\lambda_{n}\rightarrow0$. If the sequence $\left\{ N_{n}\right\} _{n=1}^{\infty}$
is sufficiently sparse, then for the operator $H=\Delta+V$, with
$V$ as defined in $($\ref{eq:Pearson_Potential}$)$, and with Neumann
boundary conditions, \eqref{eq:SineKernelAsymptotics} holds uniformly for $\xi$ in compact subsets of $\bbR^+$ and $a,b$ in compact subsets of the strip $|\textrm{Im}z| \leq 1$.
\end{theorem}

\begin{remark}
By `$\left\{ N_{n}\right\} _{n=1}^{\infty}$ is sufficiently sparse'
we mean that $N_{k+1}$ has to be chosen sufficiently large, as a
function of $N_{1},N_{2},\ldots,N_{k}$. In other words, for every
$k\geq1$, there exists a function $\widetilde{N}_{k}\left(N_{1},N_{2},\ldots,N_{k}\right)$,
such that $N_{k+1}\geq\widetilde{N}_{k}\left(N_{1},N_{2},\ldots,N_{k}\right)$.
The functions $\widetilde{N}_{k}$ will depend on $\left\{ \lambda_{n}\right\} _{n=1}^{\infty}$.
In particular, we require that $\frac{N_{k}}{N_{k+1}}\longrightarrow0$.
\end{remark}

\begin{corollary}
There exist operators with purely singular continuous spectrum on $(0,\infty)$, exhibiting
uniform clock behavior on any compact interval $I\subseteq\left(0,\infty\right)$.
\end{corollary}

\begin{proof}
Since $V(x) \rightarrow 0$ as $x \rightarrow \infty$, it is not difficult to show that the density of states for the Pearson operators considered here exists and is equal to the free density of states (e.g., by imitating the proof of \cite{geronimo}). This, together with Proposition \ref{thm:USKA-implies-clock} implies the conclusion.
\end{proof}

In rough terms, the strategy of our proof is as follows. Given $H$, we define a sequence of operators by truncating the potential
at increasing points. We refer to them as the truncated operators.
Like the original operator, $H$, they act on $L^{2}\left(\mathbb{R}^{+}\right)$. Let
\beq \label{eq:Truncated_Op}
H^{\left(\ell\right)}=\Delta+V^{\left(\ell\right)}=\Delta+\sum_{n=1}^{\ell}\lambda_{n}W\left(x-N_{n}\right).
\eeq
These operators are defined using the same boundary conditions as $H$ (and \textit{should not} be confused with
the restricted operators, $H^{L}$). Denote by $u^{\left(\ell\right)}\left(\xi,x\right)$ the unique solution of
the associated eigenfunction equation
\beq \label{eq:Truncated_Ef}
\left(\Delta+V^{\left(\ell\right)}\right)u^{\left(\ell \right)} \left(\xi, x \right)=\xi u^{\left( \ell \right)} \left(\xi,x \right).
\eeq
The associated CD kernel is
\beq \label{eq:Truncated_CD} S_{L}^{\left(\ell\right)}\left(\xi,\zeta\right)=\int_{0}^{L}u^{\left(\ell\right)}\left(\xi,r\right)u^{\left(\ell\right)}\left(\zeta,r\right)dr.
\eeq

Our strategy will be to show that since \eqref{eq:SineKernelAsymptotics} holds for any $H^{\left(\ell\right)}$, if we place the perturbations sparsely enough (i.e.\ the sequence $\left\{ N_{n}\right\} _{n=1}^{\infty}$),
then for sufficiently small $\lambda$ the renormalized CD kernel remains close enough to its sine
kernel limit. Since we want to construct the sequence $\left\{ N_{n}\right\} _{n=1}^{\infty}$ such that  \eqref{eq:SineKernelAsymptotics} holds for all $\xi \in \bbR^{+}$ (and not only in a compact interval), the main challenge will be to control the constants as we increase the size of the interval that we study. After obtaining some preliminary results in Section 2, we prove Theorem \ref{thm:Main_Result} in Section 3.

{\bf Acknowledgments} We would like to thank the referee for a careful reading and many useful suggestions.

\section{\noindent Preliminaries}


We denote by $\Phi\left(\xi,x\right)$ the solution to
\beq \label{eq:free-ev}
-\frac{\textrm{d}^{2}}{\textrm{d} x^{2}}u\left(\xi,x\right)=\xi\cdot u\left(\xi,x\right)
\eeq
with Dirichlet boundary conditions, namely $\Phi\left(\xi,0\right)=0$ and $\Phi'\left(\xi,0\right)=1$.
Similarly, we denote by $\Psi\left(\xi,x\right)$ the Neumann solution,
$\Psi\left(\xi,0\right)=1$ and $\Psi'\left(\xi,0\right)=0$. We define the transfer matrix by
\begin{equation} \label{eq:Free_Fundamental_Matrix}
T^{(0)}_{x,0}\left(\xi\right)=\begin{bmatrix}\Psi\left(\xi,x\right) & \Phi\left(\xi,x\right)\\
\Psi'\left(\xi,x\right) & \Phi'\left(\xi,x\right)
\end{bmatrix}=\begin{bmatrix}\cos\left(\sqrt{\xi}x\right) & \frac{1}{\sqrt{\xi}}\mbox{\ensuremath{\sin}\ensuremath{\ensuremath{\left(\sqrt{\xi}x\right)}}}\\
-\sqrt{\xi}\mbox{\ensuremath{\sin}\ensuremath{\ensuremath{\left(\sqrt{\xi}x\right)}}} & \cos\ensuremath{\left(\sqrt{\xi}x\right)}
\end{bmatrix}
\end{equation}
so that if $u^{(0)}\left(\xi,x\right)$ is some solution to \eqref{eq:free-ev} then
\beq \no
\begin{bmatrix}u^{(0)}\left(\xi,x\right)\\
u^{(0)}\,'\left(\xi,x\right)
\end{bmatrix}=T^{(0)}_{x,0}\left(\xi\right)\begin{bmatrix}u^{(0)}\left(\xi,0\right)\\
u^{(0)}\,'\left(\xi,0\right)
\end{bmatrix}.
\eeq
For $0\leq a\leq b$ we further define $T^{(0)}_{b,a}=T^{(0)}_{b,0}T^{(0)}_{a,0}\,^{-1}$. Note that, for any $0 \leq a \leq b$,
\beq \label{eq:Det_of_transfer_matrix}
\det\left(T^{(0)}_{b,a}\left(\xi\right)\right)=1.
\eeq

Given a Pearson potential as in \eqref{eq:Pearson_Potential} with $W$ smooth, nonnegative, and with support $\subseteq [0,1]$, recall that $u(\xi, x)$ is the unique solution of \eqref{eq:ef_equation} with Neumann boundary conditions. Using the Dirichlet solution, we may define analogously the transfer matrix for $H$ to get
\beq \no
\begin{bmatrix}u\left(\xi,x\right)\\
u\,'\left(\xi,x\right)
\end{bmatrix}=T_{x,0}\left(\xi\right)\begin{bmatrix}u\left(\xi,0\right)\\
u\,'\left(\xi,0\right)
\end{bmatrix}.
\eeq

Our analysis will rely on variation of parameters. Namely we define the functions $A_1(\xi, x)$ and $A_2(\xi, x)$ through
\begin{equation} \label{eq:Definition_of_parameters}
\begin{bmatrix}A_{1}\left(\xi,x\right)\\
A_{2}\left(\xi,x\right)
\end{bmatrix}=T^{(0)}_{x,0}\left(\xi\right)^{-1}\begin{bmatrix}u\left(\xi,x\right)\\
u\,'\left(\xi,x\right)
\end{bmatrix}
\end{equation}
i.e.
\begin{eqnarray*}
u\left(\xi,x\right) & = & A_{1}\left(\xi,x\right)\Phi\left(\xi,x\right)+A_{2}\left(\xi,x\right)\Psi\left(\xi,x\right)\\
u\,'\left(\xi,x\right) & = & A_{1}\left(\xi,x\right)\Phi'\left(\xi,x\right)+A_{2}\left(\xi,x\right)\Psi'\left(\xi,x\right).
\end{eqnarray*}

Analogously, for the truncated operators, $H^{(\ell)}$, and the associated generalized eigenfunctions, $u^{\left( \ell \right)}$, defined in \eqref{eq:Truncated_Op} and \eqref{eq:Truncated_Ef}, we define the associated transfer matrix $T_{x,0}^{\left(\ell\right)}\left(\xi\right)$, which satisfies
\beq \no
\begin{bmatrix}u^{\left(\ell\right)}\left(\xi,x\right)\\
u^{\left(\ell\right)}\,'\left(\xi,x\right)
\end{bmatrix}=T_{x,0}^{\left(\ell\right)}\left(\xi\right)\begin{bmatrix}u^{\left(\ell\right)}\left(\xi,0\right)\\
u^{\left(\ell\right)}\,'\left(\xi,0\right)
\end{bmatrix}.
\eeq
The functions $A_1^{(\ell)}$ and $A_2^{(\ell)}$ are also defined through
\begin{equation} \label{eq:Definition_of_f_l_g_l}
\begin{bmatrix}A_{1}^{\left(\ell\right)}\left(\xi,x\right)\\
A_{2}^{\left(\ell\right)}\left(\xi,x\right)
\end{bmatrix}=T^{(0)}_{x,0}\left(\xi\right)^{-1}\begin{bmatrix}u^{\left(\ell\right)}\left(\xi,x\right)\\
u^{\left(\ell\right)}\,'\left(\xi,x\right)
\end{bmatrix}.
\end{equation}
Note that for $x \geq N_{\ell}+1$, the functions $A_{1}^{\left(\ell\right)},A_{2}^{\left(\ell\right)}$ are constant
in $x$.
For later reference we note that, by \cite[Chapter 1, Theorem 8.4]{Codd1}, for any $x$ the functions $u(\xi,x),u^{(\ell)}(\xi,x)$ and therefore also $A_1(\xi,x), A_2(\xi,x),A_1^{(\ell)}(\xi,x),A_2^{(\ell)}(\xi,x)$ are entire functions of $\xi$.

Let $I_{m}=\left[\frac{1}{m},m\right]$. The following is clear:
\begin{lemma} \label{lem:Bound_For_Real_Arguments}
There is a constant $\hat{C}\in\mathbb{R}$,
such that for any $\xi\in \left[a,b\right]\subseteq\mathbb{R}^{+}$,
for any $x\in\mathbb{R}$
\[
\left|\left|T^{(0)}_{x,0}\left(\xi\right)\right|\right|,\left|\left|T^{(0)}_{x,0}\,^{-1}\left(\xi\right)\right|\right|\leq\hat{C}\cdot max\left\{ \sqrt{b},\frac{1}{\sqrt{a}}\right\} .
\]
Thus, for any interval $I_{m}$
\[
\left|\left|T^{(0)}_{x,0}\left(\xi\right)\right|\right|,\left|\left|T^{(0)}_{x,0}\,^{-1}\left(\xi\right)\right|\right|\leq\hat{C}\sqrt{m}.
\]
\end{lemma}
We need to extend this bound slightly to the complex plain, so that it holds on strips around $\bbR^{+}$.

\begin{lemma}
For any closed interval $I=\left[a,b\right]\subseteq\mathbb{R}^{+}$,
there exists a constant $M_{I}\in\mathbb{R}$, such that for any $\xi\in I$,
$0<x$ and $t\in\mathbb{R}$ with $\left|t\right|\leq1$,
\[
\left|\left|T^{(0)}_{x,0}\left(\xi+\frac{it}{x}\right)\right|\right|,\left|\left|T^{(0)}_{x,0}\,^{-1}\left(\xi+\frac{it}{x}\right)\right|\right|\leq M_{I}
\]
\end{lemma}
\begin{remark}
We consider the principal branch of the square root defined on $\bbC \setminus (-\infty, 0]$. This is not a problem since we only consider $\xi>0$.
\end{remark}

\begin{proof}
We need to uniformly bound the entries of the matrix
\[
T^{(0)}_{x,0}\left(\xi+\frac{it}{x}\right)=\begin{bmatrix}\cos\left(\sqrt{\xi+\frac{it}{x}}x\right) & \frac{1}{\sqrt{\xi+\frac{it}{x}}}\mbox{\ensuremath{\sin}\ensuremath{\ensuremath{\left(\sqrt{\xi+\frac{it}{x}}x\right)}}}\\
-\sqrt{\xi+\frac{it}{x}}\mbox{\ensuremath{\sin}\ensuremath{\ensuremath{\left(\sqrt{\xi+\frac{it}{x}}x\right)}}} & \cos\ensuremath{\left(\sqrt{\xi+\frac{it}{x}}x\right)}
\end{bmatrix}
\]
By developing the square root in a Taylor series around $\xi$, it is not hard to see that the entries are bounded at  $x\rightarrow \infty$. As for the limit $x \rightarrow 0$, we only need to estimate $\sqrt{\xi+\frac{it}{x}}\mbox{\ensuremath{\sin}\ensuremath{\ensuremath{\left(\sqrt{\xi+\frac{it}{x}}x\right)}}}$
(bounds for the other entries are obvious by continuity and the fact that the argument goes to zero). Writing
series around $x=0$,
\[
\sin\ensuremath{\left(\sqrt{\xi+\frac{it}{x}}x\right)}=\sqrt{\xi+\frac{it}{x}}x-\frac{1}{3!}\left(\sqrt{\xi+\frac{it}{x}}x\right)^{3}+o\left(x^{2}\right)
\]
we see that
\begin{eqnarray*}
\sqrt{\xi+\frac{it}{x}}\sin\ensuremath{\left(\sqrt{\xi+\frac{it}{x}}x\right)} & = & x\left(\xi+\frac{it}{x}\right)-\frac{1}{3!}x^{3}\left(\xi+\frac{it}{x}\right)^{2}+o\left(x\right)\\
 & = & it+o\left(x\right),
\end{eqnarray*}
and we are done.
\end{proof}

For the intervals $I_{m}$, we denote the constants above simply as $M_{m}=M_{I_{m}}$. From (\ref{eq:Definition_of_f_l_g_l}) we conclude that for any $\xi\in I_{m}$, and for any $t\in[-1,1]$
\begin{eqnarray}
\left|\left|\begin{bmatrix}A_{1}^{\left(\ell\right)}\left(\xi+\frac{it}{x},x\right)\\
A_{2}^{\left(\ell\right)}\left(\xi+\frac{it}{x},x\right)
\end{bmatrix}\right|\right| & \leq & M_{m}\left|\left|\begin{bmatrix}u^{\left(\ell\right)}\left(\xi+\frac{it}{x},x\right)\\
u^{\left(\ell\right)}\,'\left(\xi+\frac{it}{x},x\right)
\end{bmatrix}\right|\right|\nonumber \\
\left|\left|\begin{bmatrix}u^{\left(\ell\right)}\left(\xi+\frac{it}{x},x\right)\\
u^{\left(\ell\right)}\,'\left(\xi+\frac{it}{x},x\right)
\end{bmatrix}\right|\right| & \leq & M_{m}\left|\left|\begin{bmatrix}A_{1}^{\left(\ell\right)}\left(\xi+\frac{it}{x},x\right)\\
A_{2}^{\left(\ell\right)}\left(\xi+\frac{it}{x},x\right)
\end{bmatrix}\right|\right|.\label{eq:f=000026g_VS_u}
\end{eqnarray}
We assume from now on also that, for any $m$, $M_m \geq 1$.

The following is known as the continuous Christoffel-Darboux formula.
\begin{lemma}
\label{lem:CD-formula}If $\xi\neq\zeta$,
\begin{equation}
S_{L}\left(\xi,\zeta\right)=\frac{u\left(\xi,L\right)u'\left(\zeta,L\right)-u\left(\zeta,L\right)u'\left(\xi,L\right)}{\xi-\zeta}\label{eq:CD_Formula}
\end{equation}
and for the diagonal case,
\begin{equation}
S_{L}\left(\xi,\xi\right)=u'\left(\xi,L\right)\frac{\textrm{d}}{\textrm{d}\xi}u\left(\xi,L\right)-\frac{\textrm{d}}{\textrm{d}\xi}u'\left(\xi,L\right)u\left(\xi,L\right)\label{eq:CD_Formula_Diagonal}
\end{equation}
\end{lemma}
\begin{proof}
This is easily proved using integration by parts; see Lemma 3.4 of
\cite{maltsev}.
\end{proof}

Recall the CD kernels, $S_x^{(\ell)} \left(\xi, \zeta \right)$, associated with $H^{(\ell)}$ defined in \eqref{eq:Truncated_CD}.
Let
\beq \label{eq:tildeA}
\begin{split}
\widetilde{A_1}^{(\ell)}(\xi,x) &=\frac{A_1^{(\ell)}(\xi,x)}{\sqrt{\xi}} \\
\widetilde{A_2}^{(\ell)}(\xi,x) &=A_2^{(\ell)}(\xi,x)
\end{split}
\eeq

\begin{lemma}
\label{lem:convergence_when_divided_by_Kappa}
For any $\ell \in \bbN$ and $\xi\in \mathbb{R}^+$:
\beq \no
\lim_{x \rightarrow \infty}
\frac{S_{x}^{\left(\ell\right)}\left(\xi+\frac{a}{x},\xi+\frac{b}{x}\right)}{x\left(\frac{\widetilde{A_{1}}^{\left(\ell\right)}\left(\xi,x\right)^{2}+\widetilde{A_{2}}^{\left(\ell\right)}\left(\xi,x\right)^{2}}{2}\right)} = \frac{\sin\left(\pi\cdot\rho\left(\xi\right)\left(a-b\right)\right)}{\pi\cdot\rho\left(\xi\right)\left(a-b\right)},
\eeq
where $\rho(\xi)=\frac{1}{2\pi \sqrt{\xi}}$ is the derivative of the density of states of $\Delta$.
Moreover, for any $m$ and $C>0$, the convergence is uniform in $\xi \in I_m$ and in $|a|,|b| \leq C$. That is, for any $m, C>0$ and any $\varepsilon>0$, there exists $N(\varepsilon,m)$ so that for any $x \geq N(\varepsilon,m)$, any $\xi \in I_m$, and any  $|a|,|b| \leq C$
\[
\left|\frac{S_{x}^{\left(\ell\right)}\left(\xi+\frac{a}{x},\xi+\frac{b}{x}\right)}{x\left(\frac{\widetilde{A_{1}}^{\left(\ell\right)}\left(\xi,x\right)^{2}+\widetilde{A_{2}}^{\left(\ell\right)}\left(\xi,x\right)^{2}}{2}\right)}-\frac{\sin\left(\pi\cdot\rho\left(\xi\right)\left(a-b\right)\right)}{\pi\cdot\rho\left(\xi\right)\left(a-b\right)}\right|<\epsilon.
\]
\end{lemma}

\begin{proof}
We first prove uniform convergence for $|a|, |b| \leq C$, $|a-b| \geq \delta$ for any $\delta> 0$. For any $\xi'$ and $j=1,2$ we let
\beq \no
\widehat{A_j}(\xi')=\widehat{A_j}^{(\ell)}(\xi')=\lim_{x \rightarrow \infty}\widetilde{A_j}^{(\ell)}(\xi',x)
\eeq
(since $\widetilde{A_j}^{(\ell)}(\xi',x)$ is constant in $x$ for $x \geq N_{\ell+1}$ the limit clearly exists), and we note that for any $a$
\beq \label{eq:limit_of_variation_coeff}
\lim_{x \rightarrow \infty}\widetilde{A_j}^{(\ell)}\left(\xi+\frac{a}{x}, x \right)=\lim_{x \rightarrow \infty}\widetilde{A_j}^{(\ell)}\left(\xi+\frac{a}{x}, N_{\ell+1} \right) =\widehat{A_j}(\xi).
\eeq

Fix $\xi \in (0,\infty)$. In order to streamline the computations below, let
\beq \no
\begin{split}
\alpha_x&=\xi+\frac{a}{x}\\
\beta_x&=\xi+\frac{b}{x}
\end{split}
\eeq
Now apply Lemma \ref{lem:CD-formula} to $S_x^{(\ell)} \left(\xi+\frac{a}{x},\xi+\frac{b}{x} \right)$ and use \eqref{eq:Definition_of_f_l_g_l} to get
\beq \no
\begin{split}
& \frac{a-b}{x}S_{x}^{\left(\ell\right)}\left(\alpha_x,\beta_x \right) \\
& = A_1^{(\ell)}\left(\alpha_x,x \right)A_1^{(\ell)} \left(\beta_x,x \right) \\
&\cdot \left(\frac{\sin\left(\sqrt{\alpha_x}x \right)\cos \left(\sqrt{\beta_x}x \right)}{\sqrt{\alpha_x}}-\frac{\sin\left(\sqrt{\beta_x}x \right)\cos \left(\sqrt{
\alpha_x}x \right)}{\sqrt{\beta_x}} \right)\\
&+ A_2^{(\ell)}\left(\alpha_x,x \right)A_2^{(\ell)} \left(\beta_x,x \right) \\
&\cdot  \left(\frac{\sin\left(\sqrt{\alpha_x}x \right)\cos \left(\sqrt{\beta_x}x \right)}{\sqrt{\alpha_x}^{-1}}-\frac{\sin\left(\sqrt{\beta_x}x \right)\cos \left(\sqrt{\alpha_x}x \right)}{\sqrt{\beta_x}^{-1}} \right)\\
&+\cos\left( \sqrt{\beta_x}x\right) \cos \left(\sqrt{\alpha_x}x \right)  \\
&\cdot \left(A_1^{(\ell)}\left(\beta_x,x \right) A_2^{(\ell)}\left(\alpha_x,x \right) - A_1^{(\ell)}\left( \alpha_x,x\right)A_2^{(\ell)}\left( \beta_x,x\right)\right) \\
&+\sin\left(\sqrt{\beta_x}x \right)\sin \left(\sqrt{\alpha_x}x \right) \\
&\cdot \left(\frac{A_1^{(\ell)}\left(\beta_x,x \right) A_2^{(\ell)}\left(\alpha_x,x \right)}{\frac{\sqrt{\beta_x}}{\sqrt{\alpha_x}}} -\frac{A_1^{(\ell)}\left( \alpha_x,x\right)A_2^{(\ell)}\left( \beta_x,x\right)}{ \frac{\sqrt{\alpha_x}}{\sqrt{\beta_x}}} \right)\\
&= \widetilde{A_1}^{(\ell)}\left(\alpha_x,x \right) \widetilde{A_1}^{(\ell)} \left(\beta_x,x \right)\sqrt{\alpha_x}\sqrt{\beta_x} \\
&\cdot \left(\frac{\sin\left(\sqrt{\alpha_x}x \right)\cos \left(\sqrt{\beta_x}x \right)}{\sqrt{\alpha_x}}-\frac{\sin\left(\sqrt{\beta_x}x \right)\cos \left(\sqrt{
\alpha_x}x \right)}{\sqrt{\beta_x}} \right)\\
&+ \widetilde{A_2}^{(\ell)}\left(\alpha_x,x \right)\widetilde{A_2}^{(\ell)} \left(\beta_x,x \right) \\
&\cdot  \left(\frac{\sin\left(\sqrt{\alpha_x}x \right)\cos \left(\sqrt{\beta_x}x \right)}{\sqrt{\alpha_x}^{-1}}-\frac{\sin\left(\sqrt{\beta_x}x \right)\cos \left(\sqrt{\alpha_x}x \right)}{\sqrt{\beta_x}^{-1}} \right)+o(1)\\
\end{split}
\eeq
since $\lim_{x \rightarrow \infty} \alpha_x=\lim_{x \rightarrow \infty} \beta_x=\xi$.

For the first term, write
\beq \no
\begin{split}
& \widetilde{A_1}^{(\ell)}\left(\alpha_x,x \right) \widetilde{A_1}^{(\ell)} \left(\beta_x,x \right)\sqrt{\alpha_x}\sqrt{\beta_x} \\
&\cdot \left(\frac{\sin\left(\sqrt{\alpha_x}x \right)\cos \left(\sqrt{\beta_x}x \right)}{\sqrt{\alpha_x}}-\frac{\sin\left(\sqrt{\beta_x}x \right)\cos \left(\sqrt{
\alpha_x}x \right)}{\sqrt{\beta_x}} \right)\\
&= \widetilde{A_1}^{(\ell)}\left(\alpha_x,x \right) \widetilde{A_1}^{(\ell)} \left(\beta_x,x \right)\sqrt{\alpha_x}\sqrt{\beta_x} \sin \left((\sqrt{\alpha_x}-\sqrt{\beta_x})x \right) \\
&\cdot \left(\frac{1}{2\sqrt{\alpha_x}}+\frac{1}{2\sqrt{\beta_x}} \right) \\
&+\widetilde{A_1}^{(\ell)}\left(\alpha_x,x \right) \widetilde{A_1}^{(\ell)} \left(\beta_x,x \right)\sqrt{\alpha_x}\sqrt{\beta_x} \sin \left((\sqrt{\alpha_x}+\sqrt{\beta_x})x \right) \\
&\cdot \left(\frac{1}{2\sqrt{\alpha_x}}-\frac{1}{2\sqrt{\beta_x}} \right) \\
= &\widetilde{A_1}^{(\ell)}\left(\alpha_x,x \right) \widetilde{A_1}^{(\ell)} \left(\beta_x,x \right)\sqrt{\alpha_x}\sqrt{\beta_x} \sin \left((\sqrt{\alpha_x}-\sqrt{\beta_x})x \right) \\
&\cdot \left(\frac{1}{2\sqrt{\alpha_x}}+\frac{1}{2\sqrt{\beta_x}} \right) +o(1). \\
\end{split}
\eeq
Using
\beq \no
\lim_{x\rightarrow\infty}\left(\sqrt{\xi+\frac{a}{x}}-\sqrt{\xi+\frac{b}{x}}\right)x=\frac{a-b}{2\sqrt{\xi}}
\eeq
it now follows that
\beq \no
\begin{split}
& \lim_{x \rightarrow \infty} \widetilde{A_1}^{(\ell)}\left(\alpha_x,x \right) \widetilde{A_1}^{(\ell)} \left(\beta_x,x \right)\sqrt{\alpha_x}\sqrt{\beta_x} \\
&\cdot \left(\frac{\sin\left(\sqrt{\alpha_x}x \right)\cos \left(\sqrt{\beta_x}x \right)}{\sqrt{\alpha_x}}-\frac{\sin\left(\sqrt{\beta_x}x \right)\cos \left(\sqrt{
\alpha_x}x \right)}{\sqrt{\beta_x}} \right)\\
&=\widehat{A_1}^{(\ell)}\left(\xi \right)^2 \sqrt{\xi} \sin \left(\frac{a-b}{2\sqrt{\xi}} \right).
\end{split}
\eeq
A similar computation shows that
\beq \no
\begin{split}
&\lim_{x\rightarrow \infty} A_2^{(\ell)}\left(\alpha_x,x \right)A_2^{(\ell)} \left(\beta_x,x \right) \\
&\cdot  \left(\frac{\sin\left(\sqrt{\alpha_x}x \right)\cos \left(\sqrt{\beta_x}x \right)}{\sqrt{\alpha_x}^{-1}}-\frac{\sin\left(\sqrt{\beta_x}x \right)\cos \left(\sqrt{\alpha_x}x \right)}{\sqrt{\beta_x}^{-1}} \right)\\
&=\widehat{A_2}^{(\ell)}\left(\xi \right)^2 \sqrt{\xi} \sin \left(\frac{a-b}{2\sqrt{\xi}} \right).
\end{split}
\eeq

Recalling that $\widetilde{A_j}^{(\ell)}(\xi,x)=\widehat{A_j}^{(\ell)}(\xi)$ for any $x \geq N_{\ell+1}$, and that $\pi \rho(\xi)=(2\sqrt{\xi})^{-1}$, it follows that we have uniform convergence for $|a|, |b| \leq C$ and $|a-b| \geq \delta$ for any $\delta>0$.

Now, since $u^{\ell} (\xi, x)$ is entire in $\xi$, $f^{\ell}_x(b)=S^{\ell}_x\left(\xi+\frac{a}{x}, \xi+\frac{b}{x} \right)$, for fixed $a$, is an analytic function of $b$. Clearly, $f(b)=\frac{\sin\left(\pi\cdot\rho\left(\xi\right)\left(a-b\right)\right)}{\pi\cdot\rho\left(\xi\right)\left(a-b\right)}$ is also an analytic function of $b$. Therefore, by using the Cauchy formula, uniform convergence of $f_x^{\ell}(b)$  to $f(b)$ in the annulus $\delta \leq |a-b| \leq 1$, implies uniform convergence on the disk $|a-b|<\delta$. This then implies uniform convergence on the disk $|a-b| \leq 1$. Thus  we have uniform convergence for $|a|, |b|  \leq C$ and we are done.

\end{proof}

The lemma above shows that for any $H^{(\ell)}$ the CD kernel is close to the desired limit for sufficiently large $x$. As a first step towards understanding $H$ we want to understand the effect of the addition of $\lambda_{(\ell+1)} W(x-N_{\ell+1})$ to the potential of $H^{(\ell)}$ (i.e., going to $H^{(\ell+1)}$). The lemma below serves this purpose.

\begin{lemma} \label{lem:MatrixDiffEq}
Let $A,B:\mathbb{R}\rightarrow M_{2}\left(\mathbb{C}\right)$
be the solutions to the following matrix initial value problems
\begin{equation}
A'\left(x\right)=\begin{bmatrix}0 & 1\\
-\xi & 0
\end{bmatrix}A\left(x\right)\label{eq:Lem_17_Def_Of_A}
\end{equation}
\begin{equation}
B'\left(x\right)=\begin{bmatrix}0 & 1\\
\lambda W\left(x\right)-\xi & 0
\end{bmatrix}B\left(x\right)\label{eq:Lem_17_Def_Of_B}
\end{equation}
with the initial condition
\[
A\left(0\right)=B\left(0\right)=\begin{bmatrix}1 & 0\\
0 & 1
\end{bmatrix}
\]
where $W\left(x\right)$ is a smooth, nonnegative function with $\supp W \subseteq [0,1]$, and $\xi\in\mathbb{R}$.
Then
\begin{equation} \label{eq:Conclusion_Of_MatrixDiffEq_Lemma}
\left|\left|A \left(x\right)-B \left(x\right)\right|\right|\leq C (\xi, W)\left|\lambda\right|
\end{equation}
where $C(\xi, W)$ is a constant depending on $\xi$ and $W\left(x\right)$.
\end{lemma}

\begin{proof}
Recall the free transfer matrix, $T_{x,0}(\xi)$, defined in \eqref{eq:Free_Fundamental_Matrix}. It is straightforward to verify that \beq \label{eq:TransferMatrixDifferential}
\frac{\textrm{d}}{\textrm{d}x}T^{(0)}_{x,0}\left(\xi\right)=\begin{bmatrix}0 & 1\\
-\xi & 0
\end{bmatrix}T^{(0)}_{x,0}(\xi).
\eeq
Since $T^{(0)}_{0,0}(\xi)=\begin{bmatrix}1 & 0\\
0 & 1
\end{bmatrix}$, it follows that $A(x)=T^{(0)}_{x,0}(\xi)$.
Thus, $A(x)$ is bounded in $x$ and invertible by (\ref{eq:Det_of_transfer_matrix}).
Denote
\beq \no
S\left(x\right)=A\left(x\right)^{-1}B\left(x\right), \quad   S\left(0\right)=\textrm{Id}.
\eeq
By differentiating
$B\left(x\right)$ and by (\ref{eq:Lem_17_Def_Of_B}),
\[
B'\left(x\right)=A'\left(x\right)S\left(x\right)+A\left(x\right)S'\left(x\right)=\begin{bmatrix}0 & 1\\
\lambda W\left(x\right)-\xi & 0
\end{bmatrix}A\left(x\right)S\left(x\right)
\]
So that by (\ref{eq:Lem_17_Def_Of_A}),
\[
\begin{bmatrix}0 & 1\\
-\xi & 0
\end{bmatrix}A\left(x\right)S\left(x\right)+A\left(x\right)S'\left(x\right)=\begin{bmatrix}0 & 1\\
\lambda W\left(x\right)-\xi & 0
\end{bmatrix}A\left(x\right)S\left(x\right)
\]
and upon rearranging,
\begin{eqnarray*}
A\left(x\right)S'\left(x\right) & = & \begin{bmatrix}0 & 0\\
\lambda W\left(x\right) & 0
\end{bmatrix}A\left(x\right)S\left(x\right)\\
S'\left(x\right) & = & \lambda W\left(x\right)A^{-1}\left(x\right)\begin{bmatrix}0 & 0\\
1 & 0
\end{bmatrix}A\left(x\right)S\left(x\right).
\end{eqnarray*}

This is a linear ODE so clearly (e.g.\ by the method of successive approximations \cite{Codd2}), for any $x \in [0,1]$
\begin{equation} \label{eq:S(1)-I}
\left|\left|S\left(x\right)-I\right|\right|\leq |\lambda| \widetilde{C}(\xi, W) x \leq |\lambda| \widetilde{C}(\xi,W)
\end{equation}
where  $\widetilde{C}$ depends on $\sup_x \| T^{(0)}_{x,0}(\xi)\|$ and on $\|W\|_\infty$.
Thus writing
\beq \no
\|A(x)-B(x)\|=\|A(x)\left(I-S(x)\right)\|
\eeq
we immediately see that for $x \in [0,1]$
\beq \no
\|A(x)-B(x)\| \leq \sup_x \|T^{(0)}_{x,0}(\xi)\| \widetilde{C}(\xi,W) |\lambda|.
\eeq

For $x > 1$ we have that
\beq \no
\left(A\left(x\right)-B\left(x\right)\right)'=\begin{bmatrix}0 & 1\\
-\xi & 0
\end{bmatrix}\left( A\left(x\right)-B(x) \right)
\eeq
since $\supp W \subseteq [0,1]$ and so, from \eqref{eq:TransferMatrixDifferential} we see that
\beq \no
A(x)-B(x)=T^{(0)}_{x,1}(\xi)\left( A(1)-B(1) \right)
\eeq
so that for $x>1$
\beq \no
\| A(x)-B(x)\|\leq \sup_x \|T^{(0)}_{x,1}(\xi) \| \sup_x \|T^{(0)}_{x,0}(\xi)\| \widetilde{C}(\xi,W) |\lambda|,
\eeq
and we are done.
\end{proof}

For an interval $I_m=[1/m,m]$, we extend the definition of $M_m$ so that for all $\xi\in I_m$,
$C(\xi,W)\leq M_{m}$.

\begin{lemma}
\label{lem:Weissman_Thm}
Fix $m \in \bbN$. For any  $\xi \in I_{m}$ and for large enough values of $\ell$,
\begin{eqnarray} \label{eq:Weissman_Thm_Eq}
\left|\left|\begin{bmatrix}u^{\left(\ell+1\right)}\left(\xi,x\right)\\
u^{\left(\ell+1\right)}\,'\left(\xi,x\right)
\end{bmatrix}-\begin{bmatrix}u^{\left(\ell\right)}\left(\xi,x\right)\\
u^{\left(\ell\right)}\,'\left(\xi,x\right)
\end{bmatrix}\right|\right| & \leq & C\cdot M_m \cdot\left|\lambda_{\ell+1}\right|\cdot\left|\left|\begin{bmatrix}u^{\left(\ell\right)}\left(\xi,x\right)\\
u^{\left(\ell\right)}\,'\left(\xi,x\right)
\end{bmatrix}\right|\right|\nonumber \\
 & \leq & \widetilde{C}\cdot M_m\cdot\left|\lambda_{\ell+1}\right|\left|\left|\begin{bmatrix}u^{\left(\ell+1\right)}\left(\xi,x\right)\\
u^{\left(\ell+1\right)}\,'\left(\xi,x\right)
\end{bmatrix}\right|\right|\nonumber \\
\end{eqnarray}
where the constants $C,\,\widetilde{C}$ depend only on the function
$W$.
Similarly,
\begin{eqnarray} \label{eq:Weissman_Thm_Eq1}
\left|\left|\begin{bmatrix}A_1^{\left(\ell+1\right)}\left(\xi,x\right)\\
A_2^{\left(\ell+1\right)}\left(\xi,x\right)
\end{bmatrix}-\begin{bmatrix}A_1^{\left(\ell\right)}\left(\xi,x\right)\\
A_2^{\left(\ell\right)}\left(\xi,x\right)
\end{bmatrix}\right|\right| & \leq & C\cdot M_m^2 \cdot\left|\lambda_{\ell+1}\right|\cdot\left|\left|\begin{bmatrix}u^{\left(\ell\right)}\left(\xi,x\right)\\
u^{\left(\ell\right)}\,'\left(\xi,x\right)
\end{bmatrix}\right|\right|\nonumber \\
 & \leq & \widetilde{C}\cdot M_m^2\cdot\left|\lambda_{\ell+1}\right|\left|\left|\begin{bmatrix}u^{\left(\ell+1\right)}\left(\xi,x\right)\\
u^{\left(\ell+1\right)}\,'\left(\xi,x\right)
\end{bmatrix}\right|\right|\nonumber \\
\end{eqnarray}
and
\begin{eqnarray} \label{eq:Weissman_Thm_Eq2}
\left|\left|\begin{bmatrix} \widetilde{A_1}^{\left(\ell+1\right)}\left(\xi,x\right)\\
\widetilde{A_2}^{\left(\ell+1\right)}\left(\xi,x\right)
\end{bmatrix}-\begin{bmatrix} \widetilde{A_1}^{\left(\ell\right)}\left(\xi,x\right)\\
\widetilde{A_2}^{\left(\ell\right)}\left(\xi,x\right)
\end{bmatrix}\right|\right| & \leq & C\cdot \widetilde{M_m}^2 \cdot\left|\lambda_{\ell+1}\right|\cdot\left|\left|\begin{bmatrix}u^{\left(\ell\right)}\left(\xi,x\right)\\
u^{\left(\ell\right)}\,'\left(\xi,x\right)
\end{bmatrix}\right|\right|\nonumber \\
 & \leq & \widetilde{C}\cdot \widetilde{M_m}^2\cdot\left|\lambda_{\ell+1}\right|\left|\left|\begin{bmatrix}u^{\left(\ell+1\right)}\left(\xi,x\right)\\
u^{\left(\ell+1\right)}\,'\left(\xi,x\right)
\end{bmatrix}\right|\right|\nonumber \\
\end{eqnarray}
where $\widetilde{M_m}=M_m \cdot \max_{\xi \in I_m} \left( \sqrt{\xi}, \sqrt{\xi}^{-1} \right)=\sqrt{m}M_m$.
\end{lemma}

\begin{proof}
Fix $N_{\ell}+1\leq x_{0}\leq N_{\ell+1}$ and let $T^{(\ell)}_{x,x_0}(\xi)$,  $T^{(\ell+1)}_{x,x_0}(\xi)$ be the associated transfer matrices from $x_0$ to $x$. It is a simple computation to check that  for $x_0<x$
\beq \no
T_{x,x_{0}}^{\left(\ell\right)}\,'\left(\xi\right)=\begin{bmatrix}0 & 1\\
-\xi & 0
\end{bmatrix} T_{x,x_0}(\xi),
\eeq
\beq \no
T_{x,x_{0}}^{\left(\ell+1\right)}\,'\left(\xi\right)=\begin{bmatrix}0 & 1\\
\lambda_{\ell+1} W\left(x-N_{\ell+1}\right)-\xi & 0
\end{bmatrix}T_{x,x_{0}}^{\left(\ell+1\right)}\left(\xi\right)
\eeq

Since
\beq \no
\left|\left|\begin{bmatrix}u^{\left(\ell+1\right)}\left(\xi,x\right)\\
u^{\left(\ell+1\right)}\,'\left(\xi,x\right)
\end{bmatrix}-\begin{bmatrix}u^{\left(\ell\right)}\left(\xi,x\right)\\
u^{\left(\ell\right)}\,'\left(\xi,x\right)
\end{bmatrix}\right|\right|=
\left|\left|\left[T_{x,x_{0}}^{\left(\ell\right)}-T_{x,x_{0}}^{\left(\ell+1\right)}\right]\begin{bmatrix}u^{\left(\ell\right)}\left(\xi,x_0\right)\\
u^{\left(\ell\right)}\left(\xi,x_0\right)
\end{bmatrix}\right|\right|,
\eeq
Lemma \ref{lem:MatrixDiffEq} implies that
\beq \no
\left|\left|\begin{bmatrix}u^{\left(\ell+1\right)}\left(\xi,x\right)\\
u^{\left(\ell+1\right)}\,'\left(\xi,x\right)
\end{bmatrix}-\begin{bmatrix}u^{\left(\ell\right)}\left(\xi,x\right)\\
u^{\left(\ell\right)}\,'\left(\xi,x\right)
\end{bmatrix}\right|\right| \leq
C' M_m \left|\lambda_{\ell+1}\right|\cdot\left|\left|\begin{bmatrix}u^{\left(\ell\right)}\left(\xi,x_0\right)\\
u^{\left(\ell\right)}\,'\left(\xi,x_0\right)
\end{bmatrix}\right|\right|
\eeq
for some $C'>0$. But
\beq \no
\left|\left|\begin{bmatrix}u^{\left(\ell\right)}\left(\xi,x_0\right)\\
u^{\left(\ell\right)}\,'\left(\xi,x_0\right)
\end{bmatrix}\right|\right| \leq  \| T_{x,x_0}^{-1} \| \left \|\begin{bmatrix}u^{\left(\ell\right)}\left(\xi,x \right)\\
u^{\left(\ell\right)}\,' \left(\xi,x\right)
\end{bmatrix}\right \|
\eeq
so the first inequality follows with $C= \sup_{x} \| T_{x,x_0}^{-1} \|  C'$.

The second inequality follows from the fact that $\lambda_{n}\rightarrow0$,
so for large enough values of $\ell$, $C M_m \left|\lambda_{\ell+1}\right|<\frac{1}{2}$.
By the triangle inequality,
\begin{eqnarray*}
\left|\left|\begin{bmatrix}u^{\left(\ell\right)}\left(\xi,x\right)\\
u^{\left(\ell\right)}\,'\left(\xi,x\right)
\end{bmatrix}\right|\right|-\left|\left|\begin{bmatrix}u^{\left(\ell+1\right)}\left(\xi,x\right)\\
u^{\left(\ell+1\right)}\,'\left(\xi,x\right)
\end{bmatrix}\right|\right| & \leq & \left|\left|\begin{bmatrix}u^{\left(\ell+1\right)}\left(\xi,x\right)\\
u^{\left(\ell+1\right)}\,'\left(\xi,x\right)
\end{bmatrix}-\begin{bmatrix}u^{\left(\ell\right)}\left(\xi,x\right)\\
u^{\left(\ell\right)}\,'\left(\xi,x\right)
\end{bmatrix}\right|\right|\\
 & \leq & C\cdot M_m \cdot\left|\lambda_{\ell+1}\right|\cdot\left|\left|\begin{bmatrix}u^{\left(\ell\right)}\left(\xi,x\right)\\
u^{\left(\ell\right)}\,'\left(\xi,x\right)
\end{bmatrix}\right|\right|
\end{eqnarray*}
so for such values of $\ell$
\beq \label{eq:comparing-ell-ell+1}
\frac{1}{2}\left|\left|\begin{bmatrix}u^{\left(\ell\right)}\left(\xi,x\right)\\
u^{\left(\ell\right)}\,'\left(\xi,x\right)
\end{bmatrix}\right|\right|
\leq\left|\left|\begin{bmatrix}u^{\left(\ell+1\right)}\left(\xi,x\right)\\
u^{\left(\ell+1\right)}\,'\left(\xi,x\right)
\end{bmatrix}\right|\right|
\eeq
and we can take $\widetilde{C}=2C$. Finally, \eqref{eq:Weissman_Thm_Eq1} follows by \eqref{eq:Definition_of_f_l_g_l} and \eqref{eq:Weissman_Thm_Eq2} is immediate from the definition of $\widetilde{A_j}^{(\ell)}$.
\end{proof}


\section{Proof of Theorem \ref{thm:Main_Result}}

\begin{proof}[Proof of Theorem \ref{thm:Main_Result}]
Given $\{\lambda_n \}_{n=1}^\infty$, let $\left\{ m_n\right\} _{n=1}^{\infty}$ satisfy
$m_n\rightarrow\infty$ monotonically, and also
\begin{equation} \label{eq:Demand_on_ell}
\lim_{n \rightarrow \infty}\left|\lambda_{n}\right|\widetilde{M_{m_{n}}}^{6}\longrightarrow0
\end{equation}
where we recall that $\widetilde{M_m}=\sqrt{m} M_m$ and $M_m=M_{I_m}$. Such a subsequence exists, since $\left|\lambda_{n}\right|\longrightarrow0$
as $n\longrightarrow\infty$, so for any natural $r$, there exists
an $N_r$ such that if $N_r<n$, then $\left|\lambda_{n}\right|\leq r^{-1} \widetilde{M_{r}}^{-6}$.
Thus, we can choose $m_{1}=m_{2}=\dots=m_{N_{2}}=1$ and $m_{N_{2}+1}=m_{N_{2}+2}=\dots=m_{N_{3}}=2$
etc.

Assume we've fixed $\{N_j \}_{j=1}^\ell$. Let $\widetilde{I}_\ell=I_{m_{\ell+1}-1/\ell}$ (a closed interval in the interior of $I_{m_{\ell+1}}$).  By Lemma
\ref{lem:convergence_when_divided_by_Kappa} there exists $\widehat{N}_{\ell}$ such that for any $x \geq \widehat{N}_{\ell}$
\beq \label{eq:Convergence-of-Truncated}
\left|\frac{S_{x}^{\left(\ell\right)}\left(\xi+\frac{a}{x},\xi+\frac{b}{x}\right)}{x\left(\frac{ \widetilde{A_{1}}^{\left(\ell\right)}\left(\xi,x\right)^{2}+\widetilde{A_{2}}^{\left(\ell\right)}\left(\xi,x\right)^{2}}{2}\right)}-\frac{\sin\left(\pi\cdot\rho\left(\xi\right)\left(a-b\right)\right)}{\pi\cdot\rho\left(\xi\right)\left(a-b\right)}\right|<1/\ell
\eeq
for any $\xi \in \widetilde{I}_\ell$, and $a, b \in \mathbb{C}$ with $|a|,|b| \leq \ell$. From now on, we shall only consider $a, b \in \mathbb{C}$ with $\textrm{Im}a, \textrm{Im}b \leq 1$. By taking $\widehat{N}_\ell$ large enough we may also assume that $\textrm{Re} \left(\xi+\frac{a}{x} \right), \textrm{Re} \left(\xi+\frac{b}{x} \right) \in I_{m_{\ell+1}}$ for $\xi \in \widetilde{I}_\ell$, $x \geq \widehat{N}_\ell$. We may also assume that for such values of the parameters
\beq \label{eq:comparingAs}
\frac{1}{2} \leq \frac{\left|\widetilde{A_1}^{(\ell)}\left(\xi+\frac{a}{x},x \right) \right|^2+\left|\widetilde{A_2}^{(\ell)}\left(\xi+\frac{a}{x},x \right) \right|^2}{\left|\widetilde{A_1}^{(\ell)}\left(\xi,x \right) \right|^2+\left|\widetilde{A_2}^{(\ell)}\left(\xi,x \right) \right|^2} \leq 2.
\eeq

We will show that, as long as we pick $N_{\ell+1}\geq \widehat{N}_\ell$ inductively,
\begin{equation} \label{eq:SineKernelAsymptotics1}
\frac{S_{x}\left(\xi+\frac{a}{x},\xi+\frac{b}{x}\right)}{S_{x}\left(\xi,\xi\right)}\underset{x\rightarrow\infty}{\longrightarrow}\frac{\sin\left(\pi\cdot\rho\left(\xi\right)\left(b-a\right)\right)}{\pi\cdot\rho\left(\xi\right)\left(b-a\right)}
\end{equation}
uniformly for $\xi$ in closed intervals of $\mathbb{R}^+$ and $ a, b$ in compact subsets of the strip $|\textrm{Im}z| \leq 1$.
Our strategy will be to first prove that
\beq \label{eq:universality-off-diagonal}
\lim_{x \rightarrow \infty}
\frac{S_{x}\left(\xi+\frac{a}{x},\xi+\frac{b}{x}\right)}{x\left(\frac{A_{1} \left(\xi,x\right)^{2}\xi+A_{2}\left(\xi,x\right)^{2}}{2}\right)} = \frac{\sin\left(\pi\cdot\rho\left(\xi\right)\left(a-b\right)\right)}{\pi\cdot\rho\left(\xi\right)\left(a-b\right)},
\eeq
uniformly for complex $|a|,|b| \leq C$ with $|\textrm{Im}a|, |\textrm{Im}b| \leq 1$ and $|a-b|>\delta>0$ and then deduce \eqref{eq:SineKernelAsymptotics1} from analyticity.

Note that for  $N_{\ell+1}\leq x \leq N_{\ell+2}$ $u\left(\xi,x\right)=u^{\left(\ell+1\right)}\left(\xi,x\right)$ and  $A_i(\xi,x)=A_i^{(\ell+1)}(\xi,x)$ $i=1,2$. We claim that it is enough to show that
\beq \label{eq:Kappa_Limit}
\begin{split}
&\underset{N_{\ell+1}\leq x\leq N_{\ell+2}}{\max}\left|\frac{S_{x}^{\left(\ell\right)}\left(\xi+\frac{a}{x},\xi+\frac{b}{x}\right)}{x\left(\frac{\widetilde{A_{1}}^{\left(\ell\right)}\left(\xi,x\right)^{2}+\widetilde{A_{2}}^{\left(\ell\right)}\left(\xi,x\right)^{2}}{2}\right)}-\frac{S_{x}^{\left(\ell+1\right)}\left(\xi+\frac{a}{x},\xi+\frac{b}{x}\right)}{x\left(\frac{\widetilde{A_{1}}^{\left(\ell+1\right)}\left(\xi,x\right)^{2}+\widetilde{A_{2}}^{\left(\ell+1\right)}\left(\xi,x\right)^{2}}{2}\right)}\right|\\
& \quad \quad \longrightarrow0
\end{split}
\eeq
as $\ell\longrightarrow\infty$. This is because, assuming \eqref{eq:Kappa_Limit}, given $\varepsilon>0$ we may choose $L$ so that for any $\ell >L$ both the quantity in \eqref{eq:Kappa_Limit} and the left hand side of \eqref{eq:Convergence-of-Truncated} are smaller than $\varepsilon/2$. Now, for any $x > N_{L+1}$ it follows that $N_{\ell+1} \leq x \leq N_{\ell+2}$ for some $\ell>L$ so that
\beq \no
\frac{S_{x}\left(\xi+\frac{a}{x},\xi+\frac{b}{x}\right)}{x\left(\frac{A_{1} \left(\xi,x\right)^{2}\xi+A_{2}\left(\xi,x\right)^{2}}{2}\right)}=\frac{S_{x}^{\left(\ell+1\right)}\left(\xi+\frac{a}{x},\xi+\frac{b}{x}\right)}{x\left(\frac{\widetilde{A_{1}}^{\left(\ell+1\right)}\left(\xi,x\right)^{2}+\widetilde{A_{2}}^{\left(\ell+1\right)}\left(\xi,x\right)^{2}}{2}\right)}.
\eeq
The $\varepsilon/2$ bound on \eqref{eq:Kappa_Limit} and on \eqref{eq:Convergence-of-Truncated} then combine to show that
\beq \no
\lim_{x \rightarrow \infty}
\left| \frac{S_{x}\left(\xi+\frac{a}{x},\xi+\frac{b}{x}\right)}{x\left(\frac{A_{1} \left(\xi,x\right)^{2}\xi+A_{2}\left(\xi,x\right)^{2}}{2}\right)} - \frac{\sin\left(\pi\cdot\rho\left(\xi\right)\left(a-b\right)\right)}{\pi\cdot\rho\left(\xi\right)\left(a-b\right)} \right| < \varepsilon.
\eeq

We now proceed to prove \eqref{eq:Kappa_Limit}. Denoting $\kappa_{x}^{\left(\ell\right)}=\frac{\widetilde{A_{1}}^{\left(\ell\right)}\left(\xi,x\right)^{2}+\widetilde{A_{2}}^{\left(\ell\right)}\left(\xi,x\right)^{2}}{2}$,
we estimate
\begin{eqnarray*}
 &  & \left|\frac{S_{x}^{\left(\ell\right)}\left(\xi+\frac{a}{x},\xi+\frac{b}{x}\right)}{x\kappa_{x}^{\left(\ell\right)}}-\frac{S_{x}^{\left(\ell+1\right)}\left(\xi+\frac{a}{x},\xi+\frac{b}{x}\right)}{x\kappa_{x}^{\left(\ell+1\right)}}\right|\\
 & \leq & \left|\frac{S_{x}^{\left(\ell\right)}\left(\xi+\frac{a}{x},\xi+\frac{b}{x}\right)}{x\kappa_{x}^{\left(\ell+1\right)}}-\frac{S_{x}^{\left(\ell+1\right)}\left(\xi+\frac{a}{x},\xi+\frac{b}{x}\right)}{x\kappa_{x}^{\left(\ell+1\right)}}\right|\\
 & + & \left|\frac{S_{x}^{\left(\ell\right)}\left(\xi+\frac{a}{x},\xi+\frac{b}{x}\right)}{x\kappa_{x}^{\left(\ell\right)}}-\frac{S_{x}^{\left(\ell\right)}\left(\xi+\frac{a}{x},\xi+\frac{b}{x}\right)}{x\kappa_{x}^{\left(\ell+1\right)}}\right|\\
 & = & \left|\frac{S_{x}^{\left(\ell\right)}\left(\xi+\frac{a}{x},\xi+\frac{b}{x}\right)}{x\kappa_{x}^{\left(\ell+1\right)}}-\frac{S_{x}^{\left(\ell+1\right)}\left(\xi+\frac{a}{x},\xi+\frac{b}{x}\right)}{x\kappa_{x}^{\left(\ell+1\right)}}\right|\\
 & + & \left|\frac{S_{x}^{\left(\ell\right)}\left(\xi+\frac{a}{x},\xi+\frac{b}{x}\right)}{x\kappa_{x}^{\left(\ell\right)}}\right|\cdot\left|\frac{\kappa_{x}^{\left(\ell+1\right)}-\kappa_{x}^{\left(\ell\right)}}{\kappa_{x}^{\left(\ell+1\right)}}\right|
\end{eqnarray*}
As in the proof of Lemma \ref{lem:convergence_when_divided_by_Kappa}, we introduce the notation $\alpha_x=\xi+\frac{a}{x}$
and $\beta_x=\xi+\frac{b}{x}$; also, in the following calculation,
we will omit the second variable, which will always be $x$. By (\ref{eq:CD_Formula})
and by introducing the notation
\begin{equation}
\Delta u^{\left(\ell+1\right)}\left(\alpha_x\right)=u^{\left(\ell+1\right)}\left(\alpha_x\right)-u^{\left(\ell\right)}\left(\alpha_x\right)\label{eq:Definition_Of_Delta_U}
\end{equation}
we evaluate
\begin{equation} \nonumber
\begin{split}
&\left|\frac{S_{x}^{\left(\ell\right)}\left(\xi+\frac{a}{x},\xi+\frac{b}{x}\right)}{x\kappa_{x}^{\left(\ell+1\right)}}-\frac{S_{x}^{\left(\ell+1\right)}\left(\xi+\frac{a}{x},\xi+\frac{b}{x}\right)}{x\kappa_{x}^{\left(\ell+1\right)}}\right|\\
 & \quad = \frac{2}{a-b}\left[\widetilde{A_{1}}^{\left(\ell+1\right)}\left(\xi,x\right)^{2}+\widetilde{A_{2}}^{\left(\ell+1\right)}\left(\xi,x\right)^{2}\right]^{-1}\times\\
 & \qquad \Big(u^{\left(\ell\right)}\left(\alpha_x\right)u^{\left(\ell\right)}\,'\left(\beta_x\right)-u^{\left(\ell\right)}\left(\beta_x\right)u^{\left(\ell\right)}\,'\left(\alpha_x\right)\\
 & \qquad -u^{\left(\ell+1\right)}\left(\alpha_x\right)u^{\left(\ell+1\right)}\,'\left(\beta_x\right)+u^{\left(\ell+1\right)}\left(\beta_x\right)u^{\left(\ell+1\right)}\,'\left(\alpha_x\right)\Big)\\
 & \quad = \frac{2}{a-b}\left[\widetilde{A_{1}}^{\left(\ell+1\right)}\left(\xi,x\right)^{2}+\widetilde{A_{2}}^{\left(\ell+1\right)}\left(\xi,x\right)^{2}\right]\times\\
 & \qquad \Big(u^{\left(\ell\right)}\left(\beta_x\right)\Delta u^{\left(\ell+1\right)}\,'\left(\alpha_x\right)+\Delta u^{\left(\ell+1\right)}\left(\beta_x\right)u^{\left(\ell\right)}\,'\left(\alpha_x\right)\\
 & \qquad -u^{\left(\ell\right)}\left(\alpha_x\right)\Delta u^{\left(\ell+1\right)}\,'\left(\beta_x\right)-\Delta u^{\left(\ell+1\right)}\left(\alpha_x\right)u^{\left(\ell\right)}\,'\left(\beta_x\right)\\
 & \qquad +\Delta u^{\left(\ell+1\right)}\left(\beta_x\right)\Delta u^{\left(\ell+1\right)}\,'\left(\alpha_x\right)\\
 & \qquad -\Delta u^{\left(\ell+1\right)}\left(\alpha_x\right)\Delta u^{\left(\ell+1\right)}\,'\left(\beta_x\right)\Big).
\end{split}
\end{equation} \nonumber

By (\ref{eq:Weissman_Thm_Eq}) and (\ref{eq:f=000026g_VS_u}) (recall that $\widetilde{M_m}=M_m \max_{\xi \in I_m}\max \left(\sqrt{\xi}, \sqrt{\xi}^{-1} \right)$)
\begin{eqnarray*}
\left|\left|\begin{bmatrix}u^{\left(\ell+1\right)}\left(\alpha_x\right)\\
u^{\left(\ell+1\right)}\,'\left(\alpha_x\right)
\end{bmatrix}-\begin{bmatrix}u^{\left(\ell\right)}\left(\alpha_x\right)\\
u^{\left(\ell\right)}\,'\left(\alpha_x\right)
\end{bmatrix}\right|\right| & \leq & \widetilde{C}M_{m_{\ell}}\left|\lambda_{\ell+1}\right|\cdot\left|\left|\begin{bmatrix}u^{\left(\ell+1\right)}\left(\alpha_x\right)\\
u^{\left(\ell+1\right)}\,'\left(\alpha_x \right)
\end{bmatrix}\right|\right|\\
 & \leq & \widetilde{C}\widetilde{M_{m_{\ell}}}^{2}\left|\lambda_{\ell+1}\right|\left|\left|\begin{bmatrix}\widetilde{A_{1}}^{\left(\ell+1\right)}\left(\alpha_x, x\right)\\
\widetilde{A_{2}}^{\left(\ell+1\right)}\left(\alpha_x,x\right)
\end{bmatrix}\right|\right| \\
& \leq & 2\widetilde{C}\widetilde{M_{m_{\ell}}}^{2}\left|\lambda_{\ell+1}\right|\left|\left|\begin{bmatrix}\widetilde{A_{1}}^{\left(\ell+1\right)}\left(\xi, x\right)\\
\widetilde{A_{2}}^{\left(\ell+1\right)}\left(\xi,x\right)
\end{bmatrix}\right|\right|
\end{eqnarray*}
where the last inequality follows from \eqref{eq:comparingAs}. A similar estimate holds for $\beta_x$. Therefore
\begin{equation} \nonumber
\begin{split}
&\left|\frac{S_{x}^{\left(\ell\right)}\left(\xi+\frac{a}{x},\xi+\frac{b}{x}\right)}{x\kappa_{x}^{\left(\ell+1\right)}}-\frac{S_{x}^{\left(\ell+1\right)}\left(\xi+\frac{a}{x},\xi+\frac{b}{x}\right)}{x\kappa_{x}^{\left(\ell+1\right)}}\right| \\
& =  \frac{2}{a-b}\left|\left|\begin{bmatrix}\widetilde{A_{1}}^{\left(\ell+1\right)}\left(\xi,x\right)\\
\widetilde{A_{2}}^{\left(\ell+1\right)}\left(\xi,x\right)
\end{bmatrix}\right|\right|^{-2} \\
&\qquad \times  \Big|u^{\left(\ell\right)}\left(\beta_x\right)\Delta u^{\left(\ell+1\right)}\,'\left(\alpha_x\right) \\
&\qquad +\Delta u^{\left(\ell+1\right)}\left(\beta_x\right)u^{\left(\ell\right)}\,'\left(\alpha_x\right)\\
   & \qquad  -u^{\left(\ell\right)}\left(\alpha_x\right)\Delta u^{\left(\ell+1\right)}\,'\left(\beta_x\right) \\
   &\qquad -\Delta u^{\left(\ell+1\right)}\left(\alpha_x\right)u^{\left(\ell\right)}\,'\left(\beta_x\right)\\
 &  \qquad +\Delta u^{\left(\ell+1\right)}\left(\beta_x\right)\Delta u^{\left(\ell+1\right)}\,'\left(\alpha_x\right)\\
 & \qquad -\Delta u^{\left(\ell+1\right)}\left(\alpha_x\right)\Delta u^{\left(\ell+1\right)}\,'\left(\beta_x\right)\Big|\\
 & \leq  \frac{4}{a-b}\widetilde{C} \widetilde{M_{m_{\ell}}}^{2}\left|\lambda_{\ell+1}\right|\left|\left|\begin{bmatrix}\widetilde{A_{1}}^{\left(\ell+1\right)}\left(\xi,x\right)\\
\widetilde{A_{2}}^{\left(\ell+1\right)}\left(\xi,x\right)
\end{bmatrix}\right|\right|^{-1}\times\\
 &   \,\,\,\,\Big(2\left|u^{\left(\ell\right)}\left(\alpha_x \right)\right|+2\left|u^{\left(\ell\right)}\left(\beta_x \right)\right|+\left|u^{\left(\ell\right)}\,'\left(\alpha_x \right)\right|+\left|u^{\left(\ell\right)}\,'\left(\beta_x\right)\right|\\
 &   +\left|u^{\left(\ell+1\right)}\left(\alpha_x\right)\right|+\left|u^{\left(\ell+1\right)}\left(\beta_x\right)\right|\Big).
\end{split}
\end{equation}

We are almost done with the first term. the only real issue is that the numerator is evaluated at $\alpha_x$ and $\beta_x$, while the denominator is evaluated at $\xi$. To fix this, use \eqref{eq:Weissman_Thm_Eq} to replace $u^{(\ell+1)}$ with $u^{(\ell)}$ (for sufficiently large $\ell$) and then \eqref{eq:f=000026g_VS_u} to replace $u^{(\ell)}$ with $\widetilde{A}^{(\ell)}$ (paying a price with some fixed constants and powers of $\widetilde{M_{m_\ell}}$). Now apply \eqref{eq:comparingAs} to replace the $\alpha_x$'s and $\beta_x$'s with $\xi$ and use \eqref{eq:Weissman_Thm_Eq2} and  \eqref{eq:f=000026g_VS_u} again to show that the numerator can be bounded by (a constant times) $\|\widetilde{A}^{(\ell+1)}\|$. Since $\widetilde{M_{m_\ell}}>1$ we conclude that for $\ell$ sufficiently large
\begin{eqnarray*}
\left|\frac{S_{x}^{\left(\ell\right)}\left(\xi+\frac{a}{x},\xi+\frac{b}{x}\right)}{x\kappa_{x}^{\left(\ell+1\right)}}-\frac{S_{x}^{\left(\ell+1\right)}\left(\xi+\frac{a}{x},\xi+\frac{b}{x}\right)}{x\kappa_{x}^{\left(\ell+1\right)}}\right| & \leq & \frac{D}{a-b}\tilde{C}\left|\lambda_{\ell+1}\right|\widetilde{M_{m_{\ell}}}^{6}
\end{eqnarray*}
(where $D$ is some universal constant) so that by (\ref{eq:Demand_on_ell}), we are done with the first term.

As for the second term, $\left|\frac{S_{x}^{\left(\ell\right)}\left(\xi+\frac{a}{x},\xi+\frac{b}{x}\right)}{x\kappa_{x}^{\left(\ell\right)}}\right|\cdot\left|\frac{\kappa_{x}^{\left(\ell+1\right)}-\kappa_{x}^{\left(\ell\right)}}{\kappa_{x}^{\left(\ell+1\right)}}\right|$,
the left factor $\left|\frac{S_{x}^{\left(\ell\right)}\left(\xi+\frac{a}{x},\xi+\frac{b}{x}\right)}{x\kappa_{x}^{\left(\ell\right)}}\right|$
is bounded since, by Lemma \ref{lem:convergence_when_divided_by_Kappa}, it converges (uniformly
on $I_{m}$) to the sine kernel. Regarding the right factor,
\begin{eqnarray*}
\frac{\kappa_{x}^{\left(\ell+1\right)}-\kappa_{x}^{\left(\ell\right)}}{\kappa_{x}^{\left(\ell+1\right)}} & = & \frac{\widetilde{A_{1}}^{\left(\ell+1\right)}\left(\xi,x\right)^{2}+\widetilde{A_{2}}^{\left(\ell+1\right)}\left(\xi,x\right)^{2}-\widetilde{A_{1}}^{\left(\ell\right)}\left(\xi,x\right)^{2}-\widetilde{A_{2}}^{\left(\ell\right)}\left(\xi,x\right)^{2}}{\widetilde{A_{1}}^{\left(\ell+1\right)}\left(\xi,x\right)^{2}+\widetilde{A_{2}}^{\left(\ell+1\right)}\left(\xi,x\right)^{2}}\\
 & = & \frac{\left(\widetilde{A_{1}}^{\left(\ell+1\right)}-\widetilde{A_{1}}^{\left(\ell\right)}\right)\left(\widetilde{A_{1}}^{\left(\ell+1\right)}+\widetilde{A_{1}}^{\left(\ell\right)}\right)+\left(\widetilde{A_{2}}^{\left(\ell+1\right)}-\widetilde{A_{2}}^{\left(\ell\right)}\right)\left(\widetilde{A_{2}}^{\left(\ell+1\right)}+\widetilde{A_{2}}^{\left(\ell\right)}\right)}{\left(\widetilde{A_{1}}^{\left(\ell+1\right)}\right)^{2}+\left(\widetilde{A_{2}}^{\left(\ell+1\right)}\right)^{2}}
\end{eqnarray*}

By \eqref{eq:Weissman_Thm_Eq2}, the definition of $\widetilde{A}^{(\ell)}$, and \eqref{eq:f=000026g_VS_u}
\begin{eqnarray*}
\left|\widetilde{A_{1}}^{\left(\ell+1\right)}-\widetilde{A_{1}}^{\left(\ell\right)}\right| & \leq & \widetilde{C}\widetilde{M_{m}}^{3}\left|\lambda_{\ell+1}\right|\left|\left|\begin{bmatrix}\widetilde{A_{1}}^{\left(\ell+1\right)}\\
\widetilde{A_{2}}^{\left(\ell+1\right)}
\end{bmatrix}\right|\right|
\end{eqnarray*}
and the same bound applies to $\left|\widetilde{A_{2}}^{\left(\ell+1\right)}-\widetilde{A_{2}}^{\left(\ell\right)}\right|$.
Therefore,
\begin{eqnarray*}
\frac{\kappa_{x}^{\left(\ell+1\right)}-\kappa_{x}^{\left(\ell\right)}}{\kappa_{x}^{\left(\ell+1\right)}} & \leq & \frac{\widetilde{C}\widetilde{M_{m}}^{3}\left|\lambda_{\ell+1}\right|\left(\left|\widetilde{A_{1}}^{\left(\ell+1\right)} \right|+\left|\widetilde{A_{1}}^{\left(\ell\right)}\right|+\left|\widetilde{A_{2}}^{\left(\ell+1\right)}\right|+\left|\widetilde{A_{2}}^{\left(\ell\right)}\right|\right)}{\left|\left|\begin{bmatrix}\widetilde{A_{1}}^{\left(\ell+1\right)}\\
\widetilde{A_{2}}^{\left(\ell+1\right)}
\end{bmatrix}\right|\right|}\\
 & \leq & 8\widetilde{C}\widetilde{M_{m}}^{5}\left|\lambda_{\ell+1}\right|
\end{eqnarray*}
for $\ell$ sufficiently large (\eqref{eq:comparing-ell-ell+1}  and \eqref{eq:Definition_of_f_l_g_l} are used).
By \eqref{eq:Demand_on_ell} we are done with the second
term as well.

We have shown \eqref{eq:universality-off-diagonal} uniformly for $\xi$ in compact subsets of $\mathbb{R}^+$ and $a, b$ complex with $|\textrm{Im} a|, |\textrm{Im} b| \leq 1$, $|a|, |b| \leq C$ for any $C>0$, and $|a-b| >\delta$ for some $\delta>0$. We now repeat the argument at the end of the proof of Lemma \ref{lem:convergence_when_divided_by_Kappa}. Since for fixed $a$, all functions involved are analytic in $b$ (note that $\kappa_x^{\ell}>0$ for all $x, \xi$), the Cauchy formula implies uniform convergence also for $|a-b| \leq \delta$. This implies \eqref{eq:universality-off-diagonal} uniformly without the restriction $|a-b| >\delta$. Since the limit for $a=b=0$ is $1$, this implies \eqref{eq:SineKernelAsymptotics1} uniformly for $\xi$ in compact subsets of $\mathbb{R}^+$ and $a,b$ in compact subsets of $\mathbb{R}$ and finishes the proof.

\end{proof}


\end{document}